\documentclass[12pt]{article}

\usepackage{amsmath}
\usepackage{amssymb}
\usepackage{amsopn}
\usepackage{amsthm}
\usepackage{amscd}
\usepackage{epsfig}
\usepackage{amsfonts}
\usepackage{latexsym}
\usepackage{graphicx}

\usepackage[usenames]{color}

\usepackage[colorlinks=true,
linkcolor=webgreen,
filecolor=webbrown,
citecolor=webgreen]{hyperref}

\definecolor{webgreen}{rgb}{0,.5,0}
\definecolor{webbrown}{rgb}{.6,0,0}

\usepackage{color}
\usepackage{fullpage}
\usepackage{float}

\theoremstyle{plain}
\newtheorem{theorem}{Theorem}
\newtheorem{corollary}[theorem]{Corollary}
\newtheorem{lemma}[theorem]{Lemma}
\newtheorem{proposition}[theorem]{Proposition}

\theoremstyle{definition}
\newtheorem{definition}[theorem]{Definition}

\newtheorem{problem}[theorem]{Problem}

\theoremstyle{remark}
\newtheorem{remark}[theorem]{Remark}

\newcommand{\seqnum}[1]{\underline{#1}}

\begin{document}

\begin{center}
{\LARGE\bf  Morphisms, Symbolic sequences, and   \\
their Standard Forms
}
\vskip 1cm
\large
F.~Michel Dekking \\
DIAM \\
Delft University of Technology \\
Mekelweg 4, 2628 CD Delft \\
The Netherlands\\
\href{mailto:f.m.dekking@tudelft.nl}{\tt f.m.dekking@tudelft.nl}\\
\end{center}

\begin{abstract}
Morphisms are homomorphisms under the concatenation operation of the set of words over a finite set. Changing the elements of the finite set does not essentially change the  morphism. We propose a way to select a unique representing member out of all these morphisms. This has applications to the classification of the shift dynamical systems generated by morphisms. In a similar way, we propose the selection of a representing sequence out of the class of symbolic sequences over an alphabet of fixed cardinality. Both methods are useful for the storing of symbolic sequences in databases, like  The On-Line Encyclopedia of Integer Sequences. We  illustrate our proposals with the $k$-symbol Fibonacci sequences.
\end{abstract}

\section{Introduction}\label{Intro}

Here is a sequence on an alphabet of six letters:
$$x= 1,4,2,1,6,3,5,4,2,3,5,6,1,4,2,1,\ldots.$$
Why should one be interested in this sequence?
Well, if one projects with the letter-to-letter map $1,3,5\rightarrow 0,\;2,4,6\rightarrow 1$, one obtains the Thue-Morse sequence
$$ 0,1,1,0,1,0,0,1,1,0,0,1,0,1,1,0,\ldots.$$
This is the most famous pure morphic sequence (\seqnum{A010060},  \seqnum{A001285},  \seqnum{A010059}).
The sequence $x$ can be obtained from a 3-block morphism (see Section \ref{sec:Fib}) starting from the Thue-Morse morphism $$0\rightarrow 01,\quad 1\rightarrow 10.$$

\noindent This gives that $x$ itself is pure morphic under the square of the morphism
$$\zeta:\qquad  1\rightarrow 23,\quad 2\rightarrow 14,\quad 3\rightarrow 21,\quad4\rightarrow 56,\quad5\rightarrow 63,\quad 6\rightarrow 54.$$
It is shown by Coven, Dekking, and
Keane \cite{CDK2014} that the shift dynamical system generated by $\zeta$ is topologically isomorphic to the Thue-Morse system, and that there essentially are only 12 injective morphisms (with associated infinite sequences) that have this property.

\noindent How should  $x$ be added to the On-Line Encyclopedia of Integer Sequences?

There is a problem here. The choice of the alphabet is quite arbitrary, another popular choice is $\{0,1,2,3,4,5\}$. But even with a fixed alphabet of 6 letters, there are 6!=720  morphisms, generating as many sequences, which are essentially the same.

The ternary Thue-Morse sequence generated by $a\rightarrow abc,\: b\rightarrow ac,\: c\rightarrow b$ occurs 12 times in OEIS, exhausting all 3! possibilities of the two alphabets $\{0,1,2\}$ and $\{1,2,3\}$. It is inconceivable that this could be done for a pure morphic sequence on an alphabet of, say, 20 letters.
So what one needs is a canonical way to extract {\it one} morphism out of these huge collections.
This will be done in the next section. In Section \ref{sec:symb} we propose a simple way to standardize symbolic sequences. In Section \ref{sec:Fib} and \ref{sec:Cloitre} we illustrate our approach with a discussion of the $k$-symbol Fibonacci morphisms.

\section{Standard forms of morphisms}\label{sec:form}

\noindent Let $\theta$ be a morphism on $r$ letters.
 The {\it standard alphabet} is the set $A=\{1,2,\ldots,r\}$.

 Let $\Pi$ be a permutation on $A$. The {\it permuted} version of $\theta$ is the morphism $\theta'$ given by
 $$\theta'(a)= \Pi^{-1} \theta(\Pi(a)), \quad \mathrm{ for \,}a\in A.$$
A morphism $\theta$ is called {\it uniform} if $|\theta(a)|=|\theta(b)|$ for all $a,b\in A$, where $|w|$ denotes the length of a word $w$.

 \begin{definition} The {\it standard form} of a uniform morphism $\theta$ is obtained by permuting the letters of the standard alphabet in such a way
 that the permuted version has the smallest word $\theta'(1)\theta'(2)\cdots\theta'(r)$ in the lexicographical order on words.
\end{definition}

\noindent As an example, let $\theta$ be the Toeplitz or period doubling morphism, see

 \seqnum{A096268}: Fixed point of the morphism $0 \rightarrow 01, 1 \rightarrow 00$,

 \seqnum{A056832}: Fixed point of the morphism $1\rightarrow 12, 2\rightarrow 11$,

 \seqnum{A035263}: Trajectory of 1 under the morphism $1\rightarrow 10, 0 \rightarrow 11$.

\noindent  The second OEIS entry is on the standard alphabet. The permuted version, with $\Pi= 1 \rightleftarrows 2$ is
$1\rightarrow 22, 2\rightarrow 21$. Since $1211<2212$, the standard form of the morphism is $1\rightarrow 12, 2\rightarrow 11$.

\begin{remark} A caveat: the standard form of a morphism has some undesirable properties. For instance, the square of the standard form need not be equal to the standard form of the square.\qed
\end{remark}

\noindent Note that the standard form of a morphisms is uniquely defined, since
$$\theta'(1)\theta'(2)\cdots\theta'(r)=\theta(1)\theta(2)\cdots\theta(r)\;\Longrightarrow\; \theta'=\theta.$$
 This need not hold in the non-uniform case. For example, if $\theta$ is given by
$$\theta(1)=12,\;   \theta(2)=312,\;   \theta(3)=3123,$$
then $\theta(123)=123123123$, and the same word appears under the  version of $\theta$ permuted by $\Pi=1\rightarrow 2\rightarrow 3$.

\begin{definition} The {\it standard form} of a  morphism $\theta$ is obtained in the following way. First, permute the letters of the standard alphabet in such a way  that the permuted version $\theta'$ has the smallest word $w_{\chi}:=\theta'(1)\theta'(2)\cdots\theta'(r)$ in the lexicographical order on words. Among all permutations that have the same $w_{\chi}$, choose the one with the lexicographically smallest vector $(|\theta'(1)|,|\theta'(2)|,\ldots,|\theta'(r)|)$.
\end{definition}

\noindent  Note that once more, the standard form is uniquely defined.

The standard form of ternary Thue-Morse is the morphism
$1\rightarrow123,\; 2\rightarrow 13,\; 3\rightarrow 2.$

\section{Standard forms of symbolic sequences}\label{sec:symb}

 We now focus on infinite sequences generated by morphisms, and in particular fixed points of morphisms. These are typical examples of what we call symbolic sequences, where a change of symbols does not change the essential properties of the sequence. An example of a non-symbolic sequence on four symbols is $(s_n)$=\seqnum{A002828}, which gives the least number of squares that add up to $n$.

Without loss of generality we assume that all symbols of $A=\{1,\dots,r\}$  occur in a sequence over $A$. We then have the following very simple way to pick a representing sequence.

\begin{definition} The {\it standard form} of a  sequence $x$ over $A$ is $\Pi(x)$, where $\Pi$ is the permutation of $A$ which yields the smallest sequence in lexicographical order.
\end{definition}

For example, the standard form of the  the Toeplitz or period doubling sequence is
$1,2,1,1,1,2,1,2,1,2,1,1,1,2,1,1,\dots,$
and the standard form of the ternary Thue-Morse sequence is $1,2,3,1,3,2,1,2,3,2,1,3,1,\dots$.
It is easily seen that both sequences are the unique fixed point of the corresponding morphism, given in its standard form.
This will not happen in general. An exception, for example, is the 6-symbol Thue-Morse morphism of Section~\ref{Intro}. For archiving purposes we recommend to store both the fixed points of the standard form of the morphism and the standard form of these fixed points.

\section{$k$-symbol Fibonacci morphisms and sequences}\label{sec:Fib}

For any $N$ the $N$-block morphism $\hat{\theta}_N$ of a  morphism $\theta$ yields pure morphic sequences on an alphabet of $p(N)$ symbols, where $p(\cdot)$ is the subword complexity function of a sequence generated by $\theta$ (cf.\ \cite[p.~95]{Queff}).

In this section we take a look at another famous morphism, the  Fibonacci morphism
$$ \varphi:\qquad 0\rightarrow 01,\quad 1\rightarrow 0.$$
It is well known (see, e.g.,\cite[p.~105]{Queff}) that  $p(N)=N+1$, so  $\hat{\varphi}_N$ is a morphism on an alphabet of $N+1$ symbols.

We describe how to obtain for $N=2$ the $3$-symbol Fibonacci morphism  $\hat{\varphi}_2$.
The blocks of length 2 occurring in the language of  $\varphi$ are 00, 01, and 10.
We have
$$\varphi(00)=0101=:u_1u_2u_3u_4,\quad    \varphi(01)=010=:v_1v_2v_3, \quad    \varphi(10)=001=:w_1w_2w_3.$$
The rule is now that the length of $\hat{\varphi}_2(j_1j_2)$ is equal to the length of $\varphi(j_1)$.
So since $u_1u_2,\,u_2u_3=01,\,10$, $v_1v_2,v_2v_3=01,10$ and  $w_1w_2=00$, one obtains on the 2-blocks
$$00\rightarrow 01,\,10, \quad  01\rightarrow 01,\,10, \quad 10\rightarrow 00.$$
Coding these to the standard alphabet $\{1,2,3\}$ by $00\leftrightarrow 3, 01\leftrightarrow 1$ and $10 \leftrightarrow 2$, we obtain the standard form
$$\hat{\varphi}_2(1)=12 \quad \hat{\varphi}_2(2)=3, \quad \hat{\varphi}_2(3)=12.$$
The unique fixed point of $\hat{\varphi}_2$  is
$$x^{(2)}= 1, 2, 3, 1,2,  1, 2, 3, 1, 2, 3,1, 2,1,2,3,1,2,1,2,3,\ldots.$$
This sequence will not be found\footnote{November 2015} on OEIS.
However, the sequence obtained by mapping $1\rightarrow 2,\, 2\rightarrow 0,\, 3\rightarrow 1$ is sequence \seqnum{A159917}, and is described as the
``Fixed point of the morphism $0 \rightarrow 01,\, 1 \rightarrow 2,\,  2 \rightarrow 01$". Except for a relationship with the so-called  Fibbinary numbers
(\seqnum{A003714}), no other properties are given.
We will make a connection with the sequence $(a_n)=0, 1, 0, 2, 1, 0, 2, 0, 2,1, 0, 2, 1, 0,\ldots$ given in \seqnum{A120614}. Its definition is
\begin{equation}\label{eq:120614}
a_n=g_n-g_{n-1}\; \mathrm{where\:} g_k=\lfloor\Phi\lfloor k/\Phi\rfloor\rfloor\; \mathrm{and}\: \Phi=(1+\sqrt{5})/2.
\end{equation}
From its comments one can deduce that the sequence $1,0, 2, 1, 0, 2, 0, 2,1, 0, 2, 1, 0,\ldots$  obtained after deleting the initial entry 0 is a fixed point of the morphism $$0\rightarrow  102,\qquad 1\rightarrow 102,\qquad 2\rightarrow 02.$$

Coding this morphism to the standard alphabet via $0 \rightarrow 1, 1 \rightarrow 3, 2 \rightarrow 2$ we obtain the morphism $\theta$ given by
$$\theta(1) = 312,\quad \theta(2) = 12,\quad \theta(3) = 312.$$
We will compare $\theta$ with $\hat{\varphi}_2^2$. Note that
$$\hat{\varphi}_2^2(1) = 123,\quad \hat{\varphi}_2^2(2) = 12,\quad \hat{\varphi}_2^2(3)=123.$$
We say a morphism $\alpha$ on $A$ can be rotated if $(\alpha(a))_1 = b$ for all $a \in A$, and some fixed $b \in A$.
The { \it rotated morphism} is the morphism denoted $\rho\alpha$, defined by
$$\rho\alpha(a) = b^{-1}\,\alpha(a)\,b\quad {\rm for\: all\:} a \in A.$$
In the literature such a rotated morphism is also known as a conjugated morphism.
Note that $\rho^2 \hat{\varphi}_2^2= \theta$. This implies that $ \hat{\varphi}_2^2$
and $\theta$ generate the same language (this is not hard to prove, but  see also \cite[Lemma 3]{ABCD}). Thus arbitrarily long subwords of $x^{(2)}$ will occur in the sequence $(a_n)$ from Equation~(\ref{eq:120614}), standardized, and conversely. Actually, a little bit more is true. In the following we write shortly $\eta:=\hat{\varphi}_2^2$.

\begin{proposition}\label{prop: Fib3}
Let $\eta$ be defined by $\eta(1)=123,\; \eta(2)= 12,\; \eta(3)=123$, and let $\theta$ be defined by $\theta(1)= 312,\; \theta(2)= 12,\; \theta(3)= 312$. Then the fixed point $12312123123\cdots$ of $\eta$ equals the left shift of the fixed point $31231212312\cdots$ of $\theta$.
\end{proposition}

\begin{proof} Note first that for all $n$
$$\eta^n(2)=\eta^n(1)[\eta^{n-1}(3)]^{-1}, \quad  \theta^n(2)=[\theta^{n-1}(3)]^{-1}\theta^n(1),$$
simply because $\eta(2)=\eta(1)3^{-1}$ and $\theta(2)=3^{-1}\theta(1)$.
The proposition will be proved if we show that for all $n$
$$3\eta^n(1)=\theta^n(1)3.$$
This will be done by induction over two levels. Obviously $3\eta(1)=3123=\theta(1)3$, and $3\eta^2(1)=312312123=\theta^2(1)3$.
We have for $n\ge 2$
\begin{eqnarray*}
3\eta^{n+1}(1)&=&3\eta^n(1)\eta^n(2)\eta^n(3)=\theta^n(1)3\,\eta^n(1)[\eta^{n-1}(3)]^{-1}\,3^{-1}\theta^n(1)3\\
              &=&\theta^n(1)\,\theta^n(1)3\,[3^{-1}\theta^{n-1}(1)3]^{-1}\,3^{-1}\theta^n(1)3\\
              &=&\theta^n(3)\,\theta^n(1)\,[\theta^{n-1}(3)]^{-1}\,\theta^n(1)3\\
              &=&\theta^n(3)\,\theta^n(1)\,\theta^n(2)3=\theta^{n+1}(1) 3.
\end{eqnarray*}
\end{proof}

We end this section with 4-symbol Fibonacci. The words of length 3 in the Fibonacci language
are 001, 010, 100, and 101. Coding these to the standard alphabet $\{1,2,3,4\}$ by $001\leftrightarrow 3, 010\leftrightarrow 1, 100 \leftrightarrow 2$  and $101 \leftrightarrow 4$, we obtain the standard form morphism
$$\hat{\varphi}_3(1) = 12,\quad \hat{\varphi}_3(2) = 3,\quad \hat{\varphi}_3(3) = 14,\quad \hat{\varphi}_3(4) = 3.$$
It has just one fixed point, which is
$$x^{(3)}= 1,2,3,1,4,1,2,3,1,2,3,1,4,1,2,3,\ldots .$$
This happens to be \seqnum{A138967}, which indeed deserves to be called the Infinite Fibonacci word on the alphabet $\{1,2,3,4\}$. To prove this, note that $\hat{\varphi}_3^2$ is the morphism
$$1 \rightarrow 241,\quad 2 \rightarrow 241,\quad 3 \rightarrow 24,\quad 4 \rightarrow 24.$$
If we make this morphism injective by merging symbols with equals images (cf.\ \cite{BDM2004}), we obtain the  morphism $\bar{1}\rightarrow \bar{1}\bar{3}\bar{1}, \bar{3}\rightarrow \bar{1}\bar{3}$, which is the square of the
Fibonacci morphism. Hence, indeed, the recipe in the Comments of \seqnum{A138967}:
``replace $\bar{1}$ by 1,2,3 and $\bar{3}$ by 1,4 in the infinite 2-symbol Fibonacci word"\!, generates $x^{(3)}$.

\section{A tale of two sequences}\label{sec:Cloitre}

In the previous section we showed that the sequence $(a_{n+1})= 1, 0, 2, 1, 0, 2, 0, 2,1, 0, 2, 1, 0,\ldots$ given in \seqnum{A120614} essentially equals the 3-symbol Fibonacci sequence. Its definition is
$$a_n=g_n-g_{n-1}\; \mathrm{where\:} g_k=\lfloor\Phi\lfloor k/\Phi\rfloor\rfloor\; \mathrm{and}\: \Phi=(1+\sqrt{5})/2. $$
Here $(g_{n})$ is \seqnum{A120613} given by
$$(g_{n+1})= 1, 1, 3, 4, 4, 6, 6, 8, 9, 9, 11, 12, 12, 14, 14, 16, 17, 17, 19, 19, 21, 22, 22, \ldots.$$
This sequence was contributed to OEIS by  Cloitre in 2006, but the only reference in its description is to a sequence in the paper \cite{Zeck} by  Griffiths, a paper which occurred in 2011.
On page 502 in \cite{Zeck} Griffiths states ``Incidentally, it is interesting to note that the sequence
 $$\Big\lfloor \frac{\lfloor (n+1)\Phi\rfloor-1}{\Phi}\Big\rfloor$$
is also related to the golden string. The first few terms are given by
$1, 1, 3, 4, 4, 6, 6, 8, 9, 9, 11,$ $12, 12, 14, 14, 16, 17, 17, 19, 19, 21, 22, 22,\dots .$"

\noindent Griffiths does  not give any references to other work.

\begin{problem}  How to show that $\lfloor\Phi\lfloor\frac{n}{\Phi}\rfloor\rfloor=\Big\lfloor \frac{\lfloor n\Phi\rfloor-1}{\Phi}\Big\rfloor$ for all $n\ge 1$?
\end{problem}
	
 Using $1/\Phi=\Phi-1$, the relation is equivalent to
$$\lfloor\Phi\lfloor n\Phi-n\rfloor\rfloor=\lfloor\Phi\lfloor n\Phi\rfloor- \lfloor n\Phi\rfloor -\Phi  +1 \rfloor.$$
Taking $-n$ out of the floor, and bringing  $\lfloor n\Phi\rfloor$ to the other side this is equivalent to
\begin{equation}\label{eq:eq1}
\lfloor\Phi\lfloor n\Phi\rfloor-n\Phi+ \lfloor n\Phi\rfloor\rfloor=\lfloor\Phi\lfloor n\Phi\rfloor -\Phi  +1 \rfloor \quad {\rm for\: all\: } n\ge 1.
\end{equation}
This clearly exposes the close relationship between the two sequences:
putting $\{n\Phi\}=n\Phi-\lfloor n\Phi\rfloor$, we of course have $-1\le -\{n\Phi\}\le 0$ and $ -\Phi  +1=-0.618\dots$.
We thus see that Equation \eqref{eq:eq1} is equivalent to
\begin{equation}\label{eq:eq2}
\{\Phi\lfloor n\Phi\rfloor\}<\Phi-1 \quad\Longleftrightarrow\quad \{\Phi\lfloor n\Phi\rfloor\}<\{n\Phi\} \quad {\rm for\: all\: } n\ge 1.
\end{equation}

Let $a$ be the basic sequence \seqnum{A000201} in the papers by Carlitz et al.~in \cite{Carlitz-1}, and \cite{Carlitz-2}:
$$a(n)=\lfloor n\Phi\rfloor.$$
Then Equation \eqref{eq:eq1} is equivalent to
\begin{equation}\label{eq:eq3}
a(a(n)-n)=a(a(n)-1)-a(n)-1.
\end{equation}

But this relation will not be found in the papers by Carlitz et al. The reason is that they consequently avoid the shift operator.
For instance, if $e(n):=\lfloor (n+1)/\Phi \rfloor $
 is the fifth basic sequence defined by Carlitz et al., then $e(n)=a(n+1)-(n+1),$ but this equation does not occur.
  We remark that $(e(n))$ is \seqnum{A060143} shifted, and \seqnum{A005206} shifted (Hofstadter G-sequence). It is amusing that Equation \eqref{eq:eq3}  can be rewritten as $a(n)=a(a(n)-1)-a(a(n)-n)-1$, which is a Hofstadter-Conway type of recursion (cf.\ \seqnum{A004001}).

After this digression we will now prove the equivalence in Equation \eqref{eq:eq2} 

\begin{proposition}\label{prop:GR}
For all positive integers $n$: $$\{\lfloor n\Phi\rfloor \Phi\}<\Phi-1 \quad\Longleftrightarrow\quad \{\lfloor n\Phi\rfloor \Phi\}<\{n\Phi\}.$$
\end{proposition}

We  prove this proposition with the following lemma.

\begin{lemma} \label{lem:main}
Let $\Phi=\tfrac12(1+\sqrt{5})$ be the golden mean. Then for all positive integers $n$
$$\{\lfloor n\Phi\rfloor \Phi\}=(1-\Phi)\{n\Phi\}+1.$$
\end{lemma}

\begin{proof} All that we use is the relation $\Phi^2=\Phi+1$.
\begin{eqnarray*}
\{\lfloor n\Phi\rfloor \Phi\}&=&\lfloor n\Phi\rfloor \Phi-\lfloor \lfloor n\Phi\rfloor \Phi \rfloor\\
&=&n\Phi^2-\Phi\{n\Phi\}-\lfloor \lfloor n\Phi\rfloor \Phi \rfloor\\
&=& n\Phi+n-\Phi\{n\Phi\}-\lfloor \lfloor n\Phi\rfloor \Phi \rfloor\\
&=&\lfloor n\Phi\rfloor+\{n \Phi\}+n-\Phi\{n \Phi\}-\lfloor \lfloor n\Phi\rfloor \Phi \rfloor\\
&=& (1-\Phi)\{n \Phi\} +n+\lfloor n\Phi\rfloor-\lfloor \lfloor n\Phi\rfloor \Phi \rfloor\\
&=& (1-\Phi)\{n \Phi\}+1.
\end{eqnarray*}
Here the last step holds because the integer $n+\lfloor n\Phi\rfloor-\lfloor \lfloor n\Phi\rfloor \Phi \rfloor$ has to be equal to 1, since
$0\le\{\lfloor n\Phi\rfloor \Phi\}\le 1$, and $-0.73<1-\Phi\le (1-\Phi)\{n \Phi\} \le 0$.
\end{proof}

Proposition \ref{prop:GR} follows directly from  Lemma \ref{lem:main} by noting that the line $y=(1-\Phi)x+1$ has a negative slope, and cuts the line $y=x$
at $x=\Phi-1$.

The following well known relationship (see, e.g.,  \cite[Equation (1.11)]{Carlitz-2})  is  a corollary to the proof of Lemma \ref{lem:main}.

\begin{corollary} \label{cor:car}
For all positive integers $n$: $\lfloor \lfloor n\Phi\rfloor \Phi \rfloor=\lfloor n\Phi\rfloor+n-1.$
\end{corollary}

\bigskip
\hrule
\bigskip

\noindent{\bf{MSC}:  68R15, 37B10}
\noindent 2010 {\it Mathematics Subject Classification}: Primary 68R15; Secondary 37B10, 11B85.

\noindent \emph{Keywords: } Morphism, substitution, pure morphic, Fibonacci word, standard form.

\bigskip
\hrule

\end{document}